\documentclass[12pt,reqno,a4paper]{amsart}
\usepackage{amsmath,amssymb,amsfonts,amsthm}
\usepackage[hypertex]{hyperref} %% for DVI
\usepackage[mathscr]{eucal}
\usepackage{comment}

\date{13 (26) March  2017}

\author{Simon~Li}
\author{Ekaterina~Shemyakova}
\author{Theodore~Voronov}
\address{Department of Mathematics, SUNY  New Paltz, New Paltz, NY 12561-2443, USA}
\address{Department of Mathematics, SUNY  New Paltz, New Paltz, NY 12561-2443, USA}
\email{shemyake@newpaltz.edu}
\address{School of Mathematics, University of Manchester, Manchester, M60 1QD, UK\\
{\hphantom{hh}Dept. of Quantum Field Theory, Tomsk State University, Tomsk, 634050, Russia}}
\email{theodore.voronov@manchester.ac.uk}

\title[Differential operators on the superline]{Differential operators on the superline, Berezinians, and  Darboux transformations}

\subjclass[2010]{Primary: 16S32; % 	Rings of differential operators
secondary:
58C50,  % Analysis on supermanifolds or graded manifolds
37K35, %	Lie-Baecklund and other transformations
37K25 %	Relations with differential geometry
}

\keywords{Darboux transformation, intertwining relation, superline, Berezinian, super Wronskian, dressing transformation}

\newtheorem{theorem}{Theorem}[section]
\newtheorem{lemma}[theorem]{Lemma}%[section]

\newtheorem*{coro}{Corollary}

\theoremstyle{definition}
\newtheorem{de}[theorem]{Definition}%[section]
\newtheorem{example}[theorem]{Example}%[section]

\newtheorem{remark}[theorem]{Remark}%[section]

\renewcommand{\leq}{\leqslant}

\DeclareMathOperator{\Ker}{Ker}
\DeclareMathOperator{\Ber}{Ber}
\DeclareMathOperator{\ber}{ber}
\DeclareMathOperator{\sdet}{sdet}
 \DeclareMathOperator{\ord}{ord}

\DeclareMathOperator{\adj}{adj}
\DeclareMathOperator{\str}{str}
\DeclareMathOperator{\tr}{tr}

\DeclareMathOperator{\DO}{DO}

\newcommand{\lder}[2]{{\partial {#1}/\partial {#2}}}

\newcommand{\oder}[2]{{\frac{d {#1}}{d {#2}}}}

\newcommand{\Z}{{\mathbb Z_{2}}}

\newcommand{\p}{\partial}

\newcommand{\al}{{\alpha}}
\newcommand{\be}{{\beta}}

\newcommand{\G}{{\Gamma}}
\newcommand{\f}{{\varphi}}

\renewcommand{\l}{{\lambda}}
\newcommand{\la}{{\lambda}}
\newcommand{\x}{{\xi}}

\newcommand{\ps}{{\psi}}

\newcommand{\const}{\mathrm{const}}

\newcommand{\ft}{{\tilde f}}

\newcommand{\itt}{{\tilde\imath}}
\newcommand{\jtt}{{\tilde\jmath}}

\newcommand{\Mf}{M_{\f}}

\newcommand{\xto}[1]{{\xrightarrow{#1}}}

\newcommand{\wrom}{\mathbf{W}}
\newcommand{\wrod}{{W}}
\renewcommand{\mod}{\,\mathrm{mod}\,}

\newcommand{\jetphi}{\boldsymbol{\varphi}}
\newcommand{\boldc}{\boldsymbol{c}}

\newcommand{\sbinom}[2]{\genfrac{[}{]}{0pt}{}{#1}{#2}}

\begin{document}
\begin{abstract} We consider  differential operators on a supermanifold of dimension $1|1$. We define non-degenerate operators   as those with an invertible top coefficient in the expansion in the ``superderivative'' $D$ (which is the square root of the shift generator, the partial derivative in an even variable, with the help of an odd
indeterminate). They are remarkably similar to ordinary differential operators. We  show  that every  non-degenerate operator  can be written in terms of    `super Wronskians' (which are certain Berezinians). We apply this to   Darboux transformations (DTs),  proving that every DT   of an arbitrary non-degenerate  operator is the composition  of elementary first order transformations.  Hence every DT   corresponds to an invariant subspace  of the source operator and,  upon a choice of  basis in this subspace,   is expressed by a super-Wronskian formula. We  consider also   dressing transformations, i.e., the effect of a DT on the coefficients of the non-degenerate operator. We calculate these transformations in examples and make  some general statements.
\end{abstract}

\maketitle
%\tableofcontents

\section{Introduction}

In this paper, we consider differential operators on the superline (by which we mean a $1|1$-dimen\-sional supermanifold) and their Darboux transformations. All differential operators on the superline can be expressed in terms of the `superderivative' $D=\p_{\x}+\x\p_x$, where $x$ and $\xi$ are the even and the
odd  indeterminates, respectively, and we define non-degenerate operators as those with an invertible top coefficient of their expansion in $D$. As we observe, such operators,  though partial, behave very similarly to  ordinary differential operators. Our main results are as follows.

We analyze algebraic properties of non-degenerate differential operators on the superline and  deduce, in particular, that every such operator  can be expressed in terms of `super Wronskians', which are Berezinians or superdeterminants of `super Wro\'{n}ski' matrices.
(Super Wronskians were first introduced in~\cite{liu-manas-crum-skdv-1997}.)
This is  analogous to the classical statement for ordinary differential operators; however, it   is trickier because  {the} Berezinian is a  more complicated function than  {the} ordinary determinant; in particular,  {it is} a rational function, not a polynomial. Note that   one has to use both $\Ber$ and $\Ber^*$, where $\Ber^*\!A:=\Ber A^{\Pi}$ and  $A^{\Pi}$ denotes the parity-reversed matrix.\footnote{For $A$ invertible, $\Ber^*\! A=(\Ber A)^{-1}$; however, we need to use non-invertible matrices as well.}

We establish a complete classification of \emph{Darboux transformations}  (DTs)   of an arbitrary non-degenerate operator. (We consider monic operators for simplicity, but this is not essential.) By definition, a  Darboux transformation   $L\to L_1$, where $L$ and $L_1$ are two monic operators of the same order, is given by a monic operator $M$ satisfying    the  \emph{intertwining relation}
\begin{equation}\label{eq.intertwmm}
    ML=L_1M\,.
\end{equation}
We show that   Darboux transformations of an operator $L$    correspond to invariant subspaces of $L$ (of particular dimensions, such as $s+1|s$ or $s|s$). We conclude, therefore, that \emph{the problem of finding    Darboux transformations for a given $L$ can be regarded as a generalization of the eigenvalue problem.}

In the short note~\cite{tv:darbouxsuperline}, we showed that every Darboux transformation is the composition of elementary Darboux transformations of first order. We elaborate the proof here and show  in addition   that every Darboux transformation  can be expressed by a super-Wronskian formula, which corresponds to a choice of basis in an invariant subspace. (The Darboux transformation itself does not depend on this choice.) This  extends to differential operators on the superline the  classic   Wronskian
formulas for Darboux transformations of ordinary differential operators, due to Darboux, Crum, and Matveev.

Finally, we consider the effect of a Darboux transformation  on the coefficients of an initial operator $L$\,; this effect is sometimes called the `dressing' (or `undressing') transformation. We calculate  the dressing transformations  for arbitrary operators $L$ of orders  $n\leq  4$ in $D$, in terms of the coefficients of the intertwining operator $M$. These coefficients can be obtained by   expanding   super Wronskians. (We include a discussion of   expansion of Berezinians, where we  obtained   some new formulas for `Berezinian cofactors'  as   a by-product of our study.)  There are some general patterns concerning dressing transformations of   operators of arbitrary  order $n$  that one can notice from these examples. We present the corresponding theorem.

Recall that what is now known   as Darboux transformations   of differential operators appeared in the studies of Darboux himself~\cite{darboux:lecons1889-2}  and other classical differential geometers such as Moutard and Tzitzeica, who applied them to  {the}  geometry of surfaces. Particular examples had been known already to Euler and Laplace (see~\cite{novikov-dynnikov:1997}). These works were almost  forgotten  until  {the} 1970s. (Exceptions were    applications in quantum mechanics,  probably found by physicists independently, without knowledge of the prior work of geometers, and the re-discovery and generalization of DTs for the Sturm--Liouville operators by Crum~\cite{crum:ass}.)   A real rebirth of  Darboux transformations happened with the advent of soliton theory.   Wahlquist and Estabrook~\cite{wahl-est:baecklund1973} and Wadati--Sanuki--Konno~\cite{wadati-sanuki-konno:1975}   used   transformations of differential operators later recognized as Darboux transformations for constructing  analogs of B\"{a}cklund transformations for the famous non-linear KdV equation.
Matveev~\cite{matveev:1979}  identified these transformations with those of Darboux and Crum, introduced the  name `Darboux transformations' and  developed them into a powerful tool in integrable system theory~\cite{matveev-salle:book1991}.

{The} investigation   of Darboux transformations in the supercase was pioneered by Liu~\cite{liu-qp-darboux-for-skdv-lmp-1995} and Liu~\&~Ma\~{n}as~\cite{liu-manas-darboux-for-manin-radul-1997, liu-manas-crum-skdv-1997} for the `super Sturm--Liouville operator'
$$L=\p^2+\al D +u\,,$$
who introduced     transformations   built on   ``seed solutions'' and their iterations.  In~\cite{liu-manas-crum-skdv-1997},     the dressing transformation of the coefficients $\al$, $u$    were expressed in terms of super Wronskians.   In~\cite{li-nimmo:darboux-twist2010}, Li and Nimmo introduced analogs of elementary first order Darboux transformations   in the abstract setting of `twisted derivations' and    obtained a formula   for the  composition   in terms of quasideterminants    (thus generalizing the results of~\cite{etingof-gelfand-retakh:1997}). From these quasideterminant formulas they deduced  the super Wronskian formulas for  the  super Sturm--Liouville operator obtained in~\cite{liu-manas-crum-skdv-1997}   and extended them  to the   case where such a  formula had  not been   available.

We would like to stress that in the classic monograph~\cite{matveev-salle:book1991} and in many other works, Darboux transformations are defined in terms of   seed solutions and     Wronskian type  formulas. However, underlying these specific formulas, there is an algebraic structure that makes it possible to study DTs in a much more general sense.  To a certain extent, this structure  can be extracted from Darboux's original works, but it has not become explicit in modern literature until  {the} 1990s.

Loosely speaking, a Darboux transformation should map a differential operator to another operator `of the same form' together with a linear transformation between kernels (or arbitrary eigenspaces). This is achieved if the `old' and `new' operators are connected by an intertwining relation of the form  $ML=L_1M$ or  $NL=L_1M$.  (In modern context,   an intertwining relation with $M$ only, without
$N$, appeared, for Sturm--Liouville operators, in~\cite{shabat:1992},~\cite{veselov-shabat:dress1993}, and~\cite{bagrov-samsonov:factorization1995};       for the  `Laplacians' $-\Delta +u$, in~\cite{berest-veselov:1998, berest-veselov:2000}; an intertwining relation   with possibly different   $M$ and $N$, appeared explicitly  in~\cite{TsarevS2009}.)

Besides having  the advantage of algebraic clarity, the   abstract  algebraic  framework  based on   intertwining relations allows one to include more general types of Darboux transformations (such as the Laplace transformations of the Schr\"{o}dinger operator in 2D and their generalizations) and to search for new Darboux transformations. Therefore this framework seems to be more suitable for a new setup, such as the study of Darboux transformations on supermanifolds.   We consider here the intertwining relation  of the form~\eqref{eq.intertwmm}
with a single operator $M$   because we view the superline as a kind of 1D object, an analog of the ordinary line. Rather than  depart  from a seed solution ansatz and Wronskian type formulas, we obtain them as a result of our classification theorem. (We  consider arbitrary non-degenerate differential operators on the superline, so  the formulas of~\cite{liu-manas-crum-skdv-1997}, \cite{li-nimmo:darboux-twist2010} for the super Sturm--Liouville operator  appear as a special case.)

For comparison,  the possibility to factorize a Darboux transformation of the Sturm--Liouville operator  on the ordinary line  into elementary transformations was established   in several steps:  in~\cite[Thm. 5]{veselov-shabat:dress1993}, when the new potential   differs from the initial one  by a constant; in~\cite{shabat1995}, for  transformations of order two; and  for the general case, in~\cite{bagrov-samsonov:factorization1995} and  the follow-up paper~\cite{bagrov:samsonov:97}, see also~\cite[\S3]{samsonov:factorization1999}. For an arbitrary operator  on the line, the factorization theorem was proved only   recently:  in~\cite{adler-marikhin-shabat:2001}. The 2D analog of these results is much more complicated. Factorization of Darboux transformations for the 2D Schr\"{o}dinger operator was established in~\cite{shemy:fact}; this required developing a new algebraic apparatus. We hope to consider the `super 2D' case elsewhere.

\section{Properties of differential operators on the superline}

 Consider a $1|1$-dimensional supermanifold, which we shortly call   the `superline'. (Note that it is possible to extend the ordinary line by more than one odd variable, but we do not consider it here.) Its global structure is not important for our purposes. Let $x$ be an even coordinate and $\x$ be an odd coordinate. Functions on the superline have the form $f(x,\x)=f^0(x)+\x f^1(x)\,$.
%\begin{equation*}
 %   f(x,\x)=f^0(x)+\x f^1(x)\,.
%\end{equation*}
We should always assume that the coefficients of the expansion in $\x$ may depend on some  unspecified extra odd parameters. Therefore, for an even function $f(x,\x)$, the coefficients  $f^0(x)$ and $f^1(x)$ are even and odd respectively.  Similarly, for an odd $f(x,\x)$, the coefficient   $f^0(x)$ is odd and the coefficient $f^1(x)$ is even. (In the absence of extra odd parameters, the   odd coefficients  become of course zero.)
See Remark~\ref{rem.oddparam} below.

Denote $D=\p_{\x}+\x\p_x$, so  $D^2=\p$, $\p=\p_x$. Hence $\p_{\x}=D-\x D^2$. The ring of differential  operators  on the superline will be denoted by $\DO(1|1)$. This ring was probably first considered in~\cite{manin:radul}, where it was embedded into the ring of formal pseudodifferential operators, see below.
An arbitrary operator $A\in \DO(1|1)$  can be uniquely written as
\begin{equation*}
    A=a_0D^m+a_{1}D^{m-1}+\ldots+a_m\,,
\end{equation*}
where the coefficients $a_k$ are functions of $x,\x$ and may also depend on some `external' even or odd parameters. (See Remark~\ref{rem.oddparam}.)  We define   \emph{order} of elements of  $\DO(1|1)$ by saying that for an operator  $A$ as  above, $\ord A\leq m$. This   differs from the usual notion; e.g., $\ord D=1$, but  $\ord \p_x=2$ and $\ord \p_{\x}=2$.

\begin{remark}  The operator $D=\p_{\x}+\x\p_x$  is well-known in physics  as the simplest supersymmetry generator, and  it goes  in the literature  under different names such as `superderivative' or `covariant derivative'. It has a simple invariant characterization as one of the two normal forms of rectifiable odd vector fields~\cite{shander:vector1980} (distinguished by whether the square of the vector field vanishes or not). %See also Remark~\ref{rem.superrect}.
\end{remark}

\begin{remark} \label{rem.superrect}
For ordinary manifolds, a vector field is (locally)  rectifiable, i.e., can be
written as a partial derivative, if and only if it does not vanish at a point.
This is what is usually meant by a ``non-degeneracy'' of a vector field at a
point (and therefore on a neighborhood of this point).  The situation for
supermanifolds is subtler.  In the above-cited seminal work of
V.~N.~Shander~\cite{shander:vector1980}, it was shown that there are three
normal forms for rectifiable  vector fields. Namely, for even vector fields
it is a partial derivative such as $\lder{}{x^1}$, where $x^1$ is an even local coordinate (the same as on ordinary manifolds); but for odd vector fields it is either a partial derivative such as $\lder{}{\x^1}$, where   $\x^1$ is an odd local coordinate, or the `superderivative'  $\lder{}{\x^1}+ \x^1\lder{}{x^1}$. (The latter two cases are distinguished by the  square of the odd vector field.) The additional subtlety is   in a ``non-degeneracy'' condition that would guarantee the rectifiability. It was found that the naive extension from ordinary manifolds is not enough. If a vector field on a supermanifold does not vanish at a point (called ``weak non-degeneracy'' in~\cite{shander:vector1980}) and it is even, then it is rectifiable, i.e., takes the form  $\lder{}{x^1}$ in suitable local coordinates. So in this case there is no difference with ordinary manifolds. For an odd vector field, it was observed that ``weak non-degeneracy'' (non-vanishing at a point) does not lead to rectifiability and indeed to any classification. Such a vector field can be brought locally to the form $\lder{}{\x^1}+\x^1 Y$ where $Y$ is an arbitrary even vector field not subject to any conditions. %, hence the problem of classification of odd weakly non-degenerate vector fields is meaningless.
Two results leading from such a seeming impasse  were obtained in~\cite{shander:vector1980}. If an odd vector field $X$ is homological, i.e., satisfies $X^2=0$, then the weak non-degeneracy gives the rectifiability: locally $X$ takes the form $\lder{}{\x^1}$. On the other hand, if a stronger  ``non-degeneracy'' condition  is  imposed, namely, that the vector field as an operator on functions is locally epimorphic, then for even vector fields it is equivalent to weak non-degeneracy (and gives the usual rectifiability), while for odd vector fields it excludes homological vector fields but guarantees the rectifiability with the normal form $\lder{}{\x^1}+ \x^1\lder{}{x^1}$. Below we define   non-degenerate differential operators on the superline (Definition~\ref{def.nondeg}). From this
viewpoint,  they
can be connected with the notion of a non-degenerate vector field in the sense of~\cite{shander:vector1980}.
Another important observation of~\cite{shander:vector1980} was that for supermanifolds, the role of ordinary differential equations should be taken by a particular type of partial differential equations,  
with the ``$1|1$-dimensional time'' and the operator $\lder{}{\tau}+\tau \lder{}{t}$ with even $t$ and odd $\tau$ replacing the ordinary time derivative $d/{dt}$. In the present paper, we obtain in a certain sense ``dual'' results:  as we shall see, the non-degenerate (partial) differential operators on the superline, as we define them, possess features very close to those of  ordinary differential operators on the (ordinary) line.
\end{remark}

\begin{remark}\label{rem.oddparam}
We consider all objects parameterized by an unspecified supermanifold playing the role of the base of
the family of the considered objects. In other words, the coefficients of functions, differential operators, etc., are taken from a commutative superalgebra  which is the algebra of
global functions on such a base, so that our objects are `defined' over this algebra. Varying this unspecified base leads to    the `functor of points' familiar from algebraic geometry.   On these notions  in the context of supergeometry see, e.g.,~\cite{leites:intro1980,leites:bookeng}, \cite{manin:gaugeeng}, and~\cite{deligne:andmorgan}.
\end{remark}

\begin{de}\label{def.nondeg}
We say that an operator $A$ of order $m$ is \emph{non-degenerate} if the top coefficient $a_0$ is invertible (in particular, even).
\end{de}

\begin{example} $\p_x=D^2$ is non-degenerate.
\end{example}

\begin{example}  $\p_{\x}=D-\x D^2$ is degenerate (top coefficient $\x$).
\end{example}

Non-degenerate operators of even order are even, and of odd order, odd. Unlike arbitrary elements of $\DO(1|1)$, non-degenerate operators cannot be divisors of zero. The set of all non-degenerate operators is  multiplicatively closed.
If $A$ is non-degenerate, then $A=a_0\cdot B$, where $B=D^m+b_{1}D^{m-1}+\ldots +b_m$ is monic. So non-degeneracy corresponds to the classical idea of   a differential equation that can be resolved with respect to the highest derivative.

Our key observation is that, although the ring $\DO(1|1)$ contains nilpotents (because the algebra of functions contains them) and cannot be described by nice algebraic words such as `Euclidean', many important properties hold for non-degenerate operators. If  $N$ is arbitrary and $M$ is non-degenerate, it is possible to divide by $M$  with a remainder from the left and from the right: that is,   there  exist unique    $Q_1, R_1$ and  $Q_2, R_2$ such that
\begin{equation*}
    N=MQ_1+R_1\,, \quad N=Q_2M+R_2\,,
\end{equation*}
where $\ord R_1, \ord R_2< \ord M$.

The next important fact is that the solution space for a non-degenerate operator on the superline is finite-dimensional (unlike the degenerate case).

\begin{example} If $D\f=0$, then $\f=\const$. (Indeed, for $\f=\f^0(x)+\x\f^1(x)$, $D\f=\f^1(x)+\x\p_x\f^0(x)$, so the equation $D\f=0$ implies $\p_x\f^0=0$, $\f^1=0$.) The constant can be even or odd, but as basis vector in the solution space we may take $\f=1$, so $\dim\Ker D=1|0$.
\end{example}

(Compare with the equation $\p_{\x}\f=0$ where the solution space consists of all functions of $x$.)

\begin{lemma} \label{lemmma:dimker}
For a non-degenerate operator $A$ of order $m$,
\begin{align*}
    \dim\Ker A&=n+1|\,n \quad \text{if \  $m=2n+1$\,,   and }\\
    \dim\Ker A&=n|\,n \quad \text{\ \ \ \ \ if \  $m=2n$} \,.
\end{align*}
%$\dim\Ker A=n|n$ if $m=2n$ and $\dim\Ker A=n+1|n$ if $m=2n+1$.
\end{lemma}
\begin{proof}
Consider the equation
\begin{equation}\label{eq.eqofordm}
    D^m\f+a_{1}D^{m-1}\f+\ldots + a_m\f=0\,.
\end{equation}
Introduce a column vector $\jetphi$ with the coordinates $\f,D\f, \ldots,D^{m-1}\f$. Note that it is written in a non-standard format: we assume that the parities of positions alternate starting from even, so   the vector  $\jetphi$ is even for an even $\f$ and odd, for an odd $\f$. (For the notion of a matrix format, see Section~\ref{sec.berez}.)  Equation~\eqref{eq.eqofordm} can be re-written in the matrix form as
\begin{equation}\label{eq.eqmatrixform}
    D\jetphi= \G \jetphi\,,
\end{equation}
where
\begin{equation}\label{eq:gamma}
    \G=\begin{pmatrix}
         0 & 1 & 0 & \dots & 0 \\
         0 & 0 & 1 & \dots & 0 \\
         \dots & \dots & \dots & \dots & \dots \\
         0 & 0 & 0 & \dots & 1 \\
         -a_m & -a_{m-1} & -a_{m-2} & \dots & -a_{1}
       \end{pmatrix}
\end{equation}
is an odd matrix. (Note once again that the parities of rows and columns are alternating, starting from even, and the last row is of parity $m$. It is of course possible to re-write $\G$ using  standard matrix format, but it will  not be elucidating.) We may consider a  general linear equation of form~\eqref{eq.eqmatrixform} with an arbitrary odd matrix $\G$. Here $\jetphi$ is a vector-function, $\jetphi=\jetphi(x,\x)$. Write $\jetphi=\jetphi_0(x)+\x\jetphi_1(x)$ and $\G=\G_0(x)+\x\G_1(x)$; then~\eqref{eq.eqmatrixform} is equivalent to the system
\begin{equation}\label{eq.matrixsystem}
    \left\{
    \begin{aligned}
    \oder{\jetphi_0}{x}&=\bigl(\G_1+\G_0^2\bigr)\jetphi_0\,,\\
    \jetphi_1 &= \G_0\,\jetphi_0\,.
     \end{aligned}
\right.
\end{equation}
Hence everything reduces to solving an ordinary linear differential equation with the (even) matrix $\G_1+\G_0^2$. Its solution is uniquely defined by an initial condition $\jetphi_0(0)$. Therefore the dimension of the solution space of the above  system~\eqref{eq.matrixsystem}, and hence of the matrix equation~\eqref{eq.eqmatrixform}, is equal to the dimension of the space of initial conditions. In our particular case of~\eqref{eq.eqmatrixform} arising from~\eqref{eq.eqofordm}, it is precisely $n+1|\,n$ for $m=2n+1$ and $n|n$ for $m=2n$, as claimed.
\end{proof}

\begin{remark} One can extend the ring $\DO(1|1)$ by adjoining the formal inverse $D^{-1}$. Due to the relation $D^2=\p$, we have $D^{-1}=D\p^{-1}=\p^{-1}D$ or, explicitly,  $D^{-1}=\x+\p^{-1}\p_{\x}$. So adjoining   $D^{-1}$ or adjoining   $\p^{-1}$ is the same. The resulting rings $\DO(1|1)[D^{-1}]$ and $\DO(1|1)[[D^{-1}]]$ of formal pseudodifferential operators were introduced by Manin and Radul~\cite{manin:radul}. Non-degenerate operators become invertible elements in the ring $\DO(1|1)[[D^{-1}]]$.
\end{remark}

\begin{remark}[not used   in the sequel] There is a nice way of expressing symbolically the solution of a linear differential equation~\eqref{eq.eqmatrixform} with an arbitrary odd matrix $\G=\G(x,\x)$. By using the notion of ordered exponential (or multiplicative integral), we can write
\begin{equation*}
    \jetphi_0(x)=P \exp\int_{x_0}^x\!dx\,\bigl(\G_1+\G_0^2\bigr)\cdot\boldc\,,
\end{equation*}
where $\boldc$ is a constant vector (initial value). Hence
\begin{equation*}
    \jetphi(x,\x)=\exp(\x \G)\cdot P \exp\int_{x_0}^x\!dx\,\bigl(\G_1+\G_0^2\bigr)\cdot\boldc\,.
\end{equation*}
In the scalar or commutative case, $\G_0^2=0$ and we would get just the exponential of $\x\G + \int_{x_0}^x\!dx\, \p_{\x}\G=(\x+\p^{-1}\p_{\x})\G=D^{-1}\G$ (if we identify the operator of indefinite integration $\int_{x_0}^xdx$ with $\p^{-1}$). Hence in this case the solution would be $\exp(D^{-1}\G)\cdot\boldc$. In the general case, we write symbolically
\begin{equation*}
    \jetphi(x,\x)=P \exp(D^{-1}\G)\cdot \boldc\,,
\end{equation*}
where we set
\begin{equation*}
    P \exp(D^{-1}\G):=
    \exp(\x \G)\cdot P \exp\int_{x_0}^x\!dx\,\bigl(\G_1+\G_0^2\bigr)\,.
\end{equation*}
\end{remark}

Consider a non-degenerate operator $L$. From Lemma~\ref{lemmma:dimker} it follows that there  always exists an even solution of $L\f=0$. Moreover, we can assume it to be invertible (at least if we assume that it is always possible to divide by a non-zero function of $x$).

\begin{example}[used in the future]
For an   invertible function $\f$,   define  the   operator
\begin{equation*}
    \Mf:=D-D\ln \f=\f\circ D\circ \f^{-1}\,.
\end{equation*}
Then $\Mf(\f)=0$ and $\Ker\Mf$ is spanned by $\f$. Every monic first-order operator has  the form $\Mf$ for some invertible $\f$. Indeed, if $L=D+\al$, then set $\f:=\exp(-D^{-1}\al)$.
\end{example}

\begin{lemma}[``B\'{e}zout's theorem'']
Let  $L$ be an arbitrary non-degenerate operator of order $m$ and let $\f$ be   an  even   solution  of the  equation $L\f=0$. Then $L$  is divisible from the right by $\Mf$, so that $L=L'\Mf$, for a non-degenerate operator $L'$ of order $m-1$.
\end{lemma}
\begin{proof} Divide with a remainder from the right: $L=L'\Mf+r$, where $r$ is a function. Then $0=L\f=L'\Mf(\f)+r\f=r\f$, hence $r=0$\,.
\end{proof}

\begin{lemma}
Every   non-degenerate operator $L$  of order $m$ can be factorized as
\begin{equation*}
    L=a_m\cdot M_{\f_1} \ldots M_{\f_m}\,,
\end{equation*}
for some   functions $\f_1$, \ldots, $\f_m$.
\end{lemma}
\begin{proof} By induction.
\end{proof}

We see that non-degenerate operators on the superline  are similar in many aspects to ordinary differential operators. There is an important statement pushing this similarity even further, namely, the expression of a non-degenerate operator in terms of the superanalog of Wronskian. For that, we need to recall some information about Berezinians or superdeterminants.

\section{Digression:  Berezinians and their expansions}\label{sec.berez}
Information in this section is mostly known or can easily be obtained from what is known. We just present it in a form convenient for our purposes. For the Cramer rule in the supercase, see~\cite{rabin:duke} and also~\cite{tv:ber}. Cofactors were introduced in~\cite{tv:ber}, but there is no detailed discussion in the literature. Formula~\eqref{eq.cofactorhard} for an odd cofactor is   new.

Matrices of linear operators acting on vector spaces or free modules over a commutative superalgebra carry an extra piece of information, namely,   parities (taking values in $\Z$)   of their rows and columns. These are the parities of the corresponding basis vectors.  Labeling matrix rows and columns by parities is called a \emph{(super)matrix format}~\cite{berezin:antieng, leites:intro1980, leites:bookeng, manin:gaugeeng}. A matrix labeled in this way is sometimes referred to as a  \emph{supermatrix}. We usually suppress the prefix. For an $n$ by $m$ matrix, %to be regarded as a supermatrix,
an introduction of a matrix format means a partition of the set of its $n$ rows into $r$ `even' rows and $s$ `odd' rows, where $n=r+s$, and a partition of the set of the $m$ columns into $p$ `even' columns and $q$ `odd' columns, where $m=p+q$. (We stress that these are but labels attached to the indices and  a priori they have nothing to do with the parities of the matrix entries.) The matrix is called \emph{even} if it has even entries in the even-even and odd-odd positions, and odd entries in the even-odd and odd-even positions (and it is called~\emph{odd} in the opposite case). It is always possible to (uniquely) transform any matrix into \emph{standard format} where all even rows go first and all odd rows go after them, and the same for columns. Matrices not in standard format   naturally occur in practical examples (such as the Wro\'{n}ski matrix in the next section). Bringing the matrix to the standard format does not change a relative order within   rows or columns of the same parities.

\emph{Berezinian} (the super analog of determinant) was discovered by F.~A.~Berezin and is now named after him. Notation:  $\Ber A$, alternative notations: $\ber A$  or  $\sdet A$. It is an essentially unique, up to taking powers, multiplicative function on the space of even invertible matrices:
\begin{equation}\label{eq.bermult}
    \Ber (AB) =\Ber A \cdot\Ber B\,.
\end{equation}
It is convenient to define $\Ber A$ axiomatically, by the following properties as a function of rows:
\begin{itemize}
  \item $\Ber A$ is homogeneous of degree $+1$ in each even row;
  \item $\Ber A$ is homogeneous of degree $-1$ in each odd row;
  \item $\Ber A$ is unchanged under elementary row transformations;
  \item $\Ber E=1$ for an identity matrix $E$.
\end{itemize}
(A similar characterization is possible in terms of columns. The axioms basically follow from the required multiplicativity property and the condition that $\Ber$ should reduce to ordinary determinant in the purely even case.) Here `homogeneous of degree $\pm 1$' means that multiplying a row of $A$ by an invertible factor $\la$ results in the multiplication of $\Ber A$ by the factor of $\la^{\pm 1}$. An elementary row transformation means replacing a row by the sum with a multiple of another row, $r_i\to r_i+\la r_j$ (here $\la$ does not have to be invertible and its parity is the sum of the parities of $r_i$ and $r_j$). By Gaussian elimination, one immediately arrives from these axioms at the one of the two equivalent explicit formulas (which are often taken as the definition of $\Ber$):
\begin{equation}\label{eq.berexpl}
    \Ber A=\frac{\det(A_{00}-A_{01}A_{11}^{-1}A_{10})}{\det A_{11}}=\frac{\det A_{00}}{\det(A_{11}-A_{10}A_{00}^{-1}A_{01})}\,.
\end{equation}
Here the blocks correspond to the subdivision of the rows and columns according to their parities.

Note that Berezinian is a rational function, not a polynomial in matrix entries. That is why it is not defined on arbitrary even matrices. Initially it is required that $A$ be invertible  for $\Ber A$ to be defined. One can see that it   suffices to require only the block  $A_{11}$   be invertible.

The \emph{parity reversion} or \emph{$\Pi$-transpose} of a matrix is the change of its format so that all the parities of the rows and columns are replaced by the opposite (nothing happens with the entries themselves). Notation: $A^{\Pi}$. In terms of the block decomposition,
\begin{equation}\label{eq.pi}
    \begin{pmatrix}
      A_{00} & A_{01} \\
      A_{10} & A_{11}
    \end{pmatrix}^{\!\!\Pi}=
     \begin{pmatrix}
      A_{11} & A_{10} \\
      A_{01} & A_{00}
    \end{pmatrix}\,.
\end{equation}
If $A$ is an even invertible matrix, both $\Ber A$ and $\Ber A^{\Pi}$ make sense and
\begin{equation}
    \Ber A^{\Pi}= (\Ber A)^{-1}\,.
\end{equation}
Following~\cite{rabin:duke}, we define the \emph{inverse Berezinian} $\Ber^*A$ by
\begin{equation}\label{eq.berpi}
    \Ber^*A:= \Ber A^{\Pi}
\end{equation}
for an arbitrary even matrix whenever $\Ber A^{\Pi}$ makes sense. The function $\Ber^*A$ is defined if and only if the block $A_{00}$ is invertible. The rational functions  $\Ber A$ and $\Ber^*A$ are mutually reciprocal on invertible even matrices; in general, it may happen that one is defined and the other is not.

The following remarkable property holds: the Berezinian $\Ber A$, as a function of rows, is multilinear in all even rows of $A$. Likewise, the inverse Berezinian $\Ber^*A$ is multilinear in all odd rows of $A$.  The same holds with respect to columns: $\Ber A$ is multilinear in all even columns of $A$ and $\Ber^*A$ is multilinear in all odd columns of $A$. (Note that rows are multiplied by scalars on the left and columns, on the right.)

From the linearity it follows that it is possible to extend the definition of $\Ber A$ and consider Berezinian of a `wrong' in the sense of parity matrix $A$, namely, a matrix obtained from an even matrix by replacing one (only one!) even row  by an odd row-vector (i.e., with odd entries on the even positions and even entries on the odd positions). The same holds for columns (when one even column is replaced by an odd column-vector). In the same way, $\Ber^* A$ makes sense for a `wrong' matrix where  an odd row or an odd column is  replaced by an even vector. Note that so defined `wrong' matrices are neither even nor odd. Berezinians of `wrong' matrices take odd values. Practically, they are calculated by the same explicit formulas (no ambiguity arises).

\begin{remark}
Warning: one  should be careful with the properties of Berezinians for the case of `wrong' arguments. If an odd vector stands in $A$ at an even position, the Berezinian is still invariant under elementary transformations provided multiples of `correct' vectors are added to the `wrong' vector, not the other way round. Otherwise one can run into a contradiction.
\end{remark}

\begin{example} Let $x$ be even, $\xi$ be odd. The square matrix of order $1|1$
\begin{equation*}
    \begin{pmatrix}
      \xi & x \\
      \xi & x
    \end{pmatrix}
\end{equation*}
is a `wrong' matrix (the odd vector $(\xi, x)$ stands in an even row; we assume standard format). We have
\begin{equation*}
    \Ber \begin{pmatrix}
      \xi & x \\
      \xi & x
    \end{pmatrix}=\frac{\xi-xx^{-1}\xi}{x}=0\,.
\end{equation*}
We can subtract the second row from the first, but not the first from the second.
\end{example}

Suppose $A=(a_{ij})$ is an even square matrix of order $n|m$. The \emph{cofactors} $\adj_{ij}A$ and $\adj_{ij}^*A$ of a matrix $A$ are defined as follows (see~\cite{tv:ber}\footnote{There is a small difference in notation between   the one used here and that   in~\cite{tv:ber}.}). When the index $i$ is even, $\adj_{ij}A$  is defined as the Berezinian of the matrix obtained from $A$ by replacing its $i$th row by a row of zeros except for the $j$th place where $1$ is inserted. Likewise, $\adj_{ij}A$ is defined when $j$ is even, by replacing   the $j$th column by a column of zeros except for $1$ at the $i$th position. Note that when both   $i$ and $j$ are  even, these definitions agree.  In the same way, one defines $\adj_{ij}^*A$, for the case where $i$ is odd or $j$ is odd and by using $\Ber^*$ instead of $\Ber$.   By the definition, there are the following expansion formulas: for  a  fixed even $i_0$ and a  fixed even $j_0$,
\begin{align}\label{eq.berrowexpan}
    \Ber A&=\sum_{j}a_{i_0j}\cdot\adj_{i_0j}A \qquad\quad \text{(row expansion)}\\
    &= \sum_{i}\adj_{ij_0}A\cdot   a_{ij_0} \qquad\quad \text{(column expansion)}\,.
\end{align}
Similarly, for a fixed odd $i_0$ and a  fixed odd $j_0$,
\begin{align}\label{eq.berstarrowexpan}
    \Ber^* A&=\sum_{j}a_{i_0j}\cdot\adj_{i_0j}^*A \qquad\quad \text{(row expansion)}\\
    &= \sum_{i}\adj_{ij_0}^*A\cdot a_{ij_0}\qquad\quad \text{(column expansion)}\,.
\end{align}
Note that $\adj_{ij}A$ is undefined when $i$, $j$ are both odd and $\adj_{ij}^*A$ is undefined when $i$, $j$ are both even.

By combining the obvious identities $\adj_{ij}A=(-1)^{\itt}\lder{\Ber A}{a_{ij}}$ for an even $i$ or even $j$ and $\adj_{ij}^*A=(-1)^{\itt+1}\lder{\Ber^* A}{a_{ij}}$ for an odd $i$ or odd $j$ with the following formula of differentiation of Berezinian:
\begin{equation}
    \delta (\Ber A) = \Ber A \, \str(\delta A\cdot A^{-1})
\end{equation}
(see, e.g.,~\cite{tv:ber}; here $\delta$ is an arbitrary even derivation and $\str$ denotes   \emph{supertrace},  $\str B=\tr B_{00}-\tr B_{11}$ for an even matrix $B$), one deduces the following relations between the cofactors of $A$ and the elements of the inverse matrix $A^{-1}=(a^*_{ij})$\,:
\begin{equation}\label{eq.inverse}
    \frac{\adj_{ij}A}{\Ber A}=a_{ji}^*= \frac{\adj_{ij}^*A}{\Ber^* A}
\end{equation}
(for all values of indices where the cofactors make sense). (Above, and everywhere we use the  notation where  tilde over the symbol means the parity of the symbol.)

\begin{example} For an even square matrix $A$ of size $1|1$, with the self-explanatory notations for the matrix entries, $\Ber A=a_{00}a_{11}^{-1}-a_{01}a_{10}a_{11}^{-2}$, and we have for example $\adj_{00}A=1/a_{11}$, $\adj_{01}A=-a_{10}/a_{11}^2=\Ber\left(\begin{smallmatrix}0 & 1\\a_{10}& a_{11}\end{smallmatrix}\right)$ (a `wrong' matrix).
\end{example}

\begin{remark} When both $i$ and $j$ are even, the cofactor $\adj_{ij}A$ is up to a sign the Berezinian of the even square matrix of size $n-1|m-1$ obtained from $A$ by deleting the $i$th row and the $j$th column, i.e., a `minor' of $A$. The sign is $(-1)^{N(i)+N(j)}$, where $N(\dots)$ is the absolute number of a row or column among the rows or columns of the same parity. This case is completely analogous to the familiar case of ordinary matrices. Unlike that, if $i$ is even and  $j$ is odd (or other way round), then the cofactor $\adj_{ij}A$ is the Berezinian of a `wrong' matrix and  is odd. The non-unit matrix entries in the $j$th column cannot be replaced by zeros and the cofactor  cannot be reduced to a Berezinian of a smaller size in a naive way;  still, there is a (more complicated)  expression for it in terms of smaller size Berezinians:
\begin{equation}\label{eq.cofactorhard}
    \adj_{ij} A=-\frac{\Ber^*d_{jj}\left(A_{r_i\to r_j}\right)}{(\Ber^*d_{jj}(A))^2}
\end{equation}
for $\itt=1$ and $\jtt=0$. Here $r_i\to r_j$ means that the $i$th row is replaced by the $j$th row, and by $d_{ij}$ we denote the operation of deleting the $i$th row and the $j$th column in a matrix. Note $\Ber^*$ and not $\Ber$  in~\eqref{eq.cofactorhard}, and that the matrix in the numerator is `wrong'. For comparison, the formula for the cofactor when both $i$ and $j$ are even, mentioned above, is
\begin{equation}\label{eq.cofactoreasy}
    \adj_{ij} A=(-1)^{N(i)+N(j)} \Ber d_{ij} (A)\,,
\end{equation}
for $\itt=\jtt=0$. Similar expressions, with interchanging $\Ber$ and $\Ber^*$, can be written for $\adj_{ij}A$. As explained, formula~\eqref{eq.cofactoreasy} is obtained directly; but both formulas~\eqref{eq.cofactoreasy} and~\eqref{eq.cofactorhard} can obtained together by writing down a system of linear equations for the coefficients $c_i=\adj_{ij}A$ (with fixed $j$, $\jtt=0$) and solving it by using Cramer's rule. Then the two different cases come about naturally. (We omit the details.)
\end{remark}

Finally, from the expansion formulas~\eqref{eq.berrowexpan} and \eqref{eq.berstarrowexpan} we obtain the following ``row by row'' rule for differentiation of Berezinian, generalizing the rule for ordinary determinants:
\begin{equation}\label{eq.difber}
    \delta (\Ber A)=\sum_{\text{$i$ even}}\Ber \begin{pmatrix}
                                      \dots \\
                                      \delta(A)_i\\
                                       \dots
                                    \end{pmatrix}-
    (\Ber A)^2 \cdot \sum_{\text{$i$ odd}}\Ber^*\begin{pmatrix}
                                      \dots \\
                                      \delta(A)_i\\
                                       \dots
                                    \end{pmatrix}\,.
\end{equation}
Here $\delta$ is some derivation (may be odd), $(A)_i$ denotes the $i$th row of a matrix, and the matrices in both sums are obtained from $A$ by replacing $(A)_i$ by the row of derivatives $\delta(A)_i$. A similar formula holds for $\delta(\Ber^*A)$, with  $\Ber$ and $\Ber^*$ interchanged; also   with  columns instead of rows (in the latter case, with extra signs).

\section{Expression of a differential operator via super Wronskians}\label{sec.swron}
In this section, we establish an  analog for the superline of the classic  relation between ordinary differential operators and Wronskian determinants.

Given a  function $f$ of variables $x,\x$, we define its \emph{superjet} %(not to be confused with the Russian airplane)
as the infinite  sequence $f,Df,D^2f, \ldots \,$. If we terminate the sequence at $D^nf$, we  speak about the \emph{$n$-superjet}. If $f$ is even, its $n$-superjet is a point (an even vector) in the vector \emph{$n$-superjet  space}, which has   dimension $k+1|k$ if $n=2k$ or $k+1|k+1$ if $n=2k+1$.  For an odd function $f$, its $n$-superjet is  an odd vector  in this  space  (a point or an even vector  in the reversed parity space,   of dimension $k|k+1$ if $n=2k$ and $k+1|k+1$ if $n=2k+1$). Without the name, we used superjets in the proof of Lemma~\ref{lemmma:dimker}.  If we use the natural numbering of the coordinates of a superjet (corresponding to the successive powers of $D$), then their parities alternate. If needed, the coordinates can be renumbered so that all even coordinates go first and all odd coordinates go second. Note also that it is natural to write superjets as column vectors (because   the multiplication of a function by a constant from the right induces the multiplication of its superjet from the right and column vectors make naturally a right module). For typographic reasons we may still write them horizontally, but use square rather than round brackets.

\begin{example} For the $2$-superjet of an even function $\f$, we have $[\f^0,\f^1,\f^2]=[\f,D\f,D^2\f]$ in the natural numbering, where at the $i$th position stands an element of parity $i\mod 2$. Or we may write $[\f,D^2\f\,|\,D\f]$, where the first two coordinates are even and the last coordinate is odd.
\end{example}

\begin{de}
Consider functions $\f_0, \f_1, \ldots \f_n$, where $\f_0$ is even, $\f_1$ is odd, etc.
The \emph{Wro\'{n}ski matrix} of such a system, notation:
\begin{equation*}
    \wrom(\f_0, \f_1, \ldots \f_n)=[W^i{}_j(\f_0, \f_1, \ldots \f_n)]_{ij,=0\ldots n}\,,
\end{equation*}
has entries $W^i{}_j(\f_0, \f_1, \ldots \f_n)=D^i\f_j$. We define the \emph{(super)wronskian} of a system $\f_0, \f_1, \ldots \f_n$ as the Berezinian of the Wro\'{n}ski matrix:
\begin{equation*}
    \wrod(\f_0, \f_1, \ldots \f_n):=\Ber\wrom(\f_0, \f_1, \ldots \f_n)
\end{equation*}
Define also the \emph{inverse (super)wronskian} of a system $\f_0, \f_1, \ldots \f_n$ by
\begin{equation*}
    \wrod^*(\f_0, \f_1, \ldots \f_n):=\Ber^*\wrom(\f_0, \f_1, \ldots \f_n)\,,
\end{equation*}
as the inverse Berezinian  of the Wro\'{n}ski matrix.
\end{de}

Superwronskians  were   introduced by Liu~\&~Ma\~{n}as~\cite{liu-manas-crum-skdv-1997}. We shall often drop the prefix and refer to   them as  simply `Wronskians'. It is clear that it is not the usual Wronskian of functions of a single variable.

The columns of the matrix $\wrom(\f_0, \f_1, \ldots \f_n)$ are the $n$-superjets of functions $\f_0$, $\f_1$, \ldots, $\f_n$. Note that the Wro\'{n}ski matrix has non-standard format, where the parity of an $i$th row or column is $i\mod 2$.
Recall from the previous section that the Berezinian of a matrix of a non-standard format is calculated by the usual formula where it is understood that the blocks are obtained  by extracting from the matrix  the   rows and  columns of the required parities with their   order preserved  (effectively, re-numbering of the rows and columns so to make the format standard, but without changing the order of rows and columns of the same parity).

\begin{example} Below we use horizontal and vertical lines to show partition of matrices  into blocks according to parity.
\begin{equation*}
    \wrod(\f_0)=\f_0\,,
\end{equation*}
\begin{equation*}
    \wrod(\f_0,\f_1)=\Ber\begin{pmatrix}
                    \f_0   &   \f_1 \\
                   D\f_0    &   D\f_1\\
                  \end{pmatrix}=
    \Ber\begin{pmatrix}
                    \f_0   & \vline & \f_1 \\
                   \hline
                   D\f_0    & \vline & D\f_1\\
                  \end{pmatrix}\,,
\end{equation*}
\begin{equation*}
    \wrod(\f_0,\f_1,\f_2)=\Ber\begin{pmatrix}
                    \f_0   & \f_1 &   \f_2 \\
                    D\f_0   & D\f_1   & D\f_2\\
                    D^2\f_0   & D^2\f_1   & D^2\f_2\\
                  \end{pmatrix}
    =
    \Ber\begin{pmatrix}
                    \f_0   & \f_2 & \vline & \f_1 \\
                    D^2\f_0   & D^2\f_2 & \vline & D^2\f_1\\
                   \hline
                   D\f_0   & D\f_2 & \vline & D\f_1\\
                  \end{pmatrix}\,,
\end{equation*}
\begin{multline*}
    \wrod(\f_0,\f_1,\f_1,\f_3)=\Ber\begin{pmatrix}
                    \f_0   & \f_1 &   \f_2 & \f_3\\
                    D\f_0   & D\f_1   & D\f_2 & D\f_3\\
                    D^2\f_0   & D^2\f_1   & D^2\f_2 & D^2\f_3\\
                    D^3\f_0   & D^3\f_1   & D^3\f_2 & D^3\f_3\\
                  \end{pmatrix}
    =\\
    \Ber\begin{pmatrix}
                    \f_0   & \f_2 & \vline & \f_1 & \f_3\\
                    D^2\f_0   & D^2\f_2 & \vline & D^2\f_1 & D^2\f_3\\
                   \hline
                   D\f_0   & D\f_2 & \vline & D\f_1 & D^2\ps_2\\
                   D^3\f_0   & D^3\f_2 & \vline & D^3\f_1 & D^3\f_3\\
                  \end{pmatrix}\,.
\end{multline*}
\end{example}

%For our purposes we shall also need the following notion.

\begin{example}
\begin{equation*}
    \wrod^*(\f_0,\f_1)=\Ber^*\begin{pmatrix}
                    \f_0   &   \f_1 \\
                   D\f_0    &   D\f_1\\
                  \end{pmatrix}=
    \Ber\begin{pmatrix}
                    D\f_1   & \vline &  D\f_0\\
                   \hline
                   \f_1    & \vline & \f_0\\
                  \end{pmatrix}\,.
\end{equation*}
\end{example}

From the properties of Berezinians it follows that \emph{the superwronskian of a system of functions is linear in each even function}, and \emph{the inverse superwronskian is linear in each odd function}.

From the invariance of Berezinians w.r.t. elementary transformations of columns, it also follows that \emph{if an even function in the argument of a superwronskian  is replaced by a linear combination (with constant coefficients of required parities) of the remaining functions, then the superwronskian vanishes; and the same holds for inverse superwronskians and odd functions}.

Now the main statement of this section.

\begin{theorem}
\label{thm.mfromker}
A monic differential operator on the superline
\begin{equation*}
    M=D^n+a_{1}D^{n-1}+\ldots +a_n
\end{equation*}
is completely defined by its kernel (or solution space) $\Ker M$. Namely, if $\f_0,\ldots,\f_{n-1}$ is a basis in $\Ker M$, where $\widetilde{\f_i}=i\mod 2$, then for $n=2k+1$ the action of $M$ on arbitrary \emph{odd} function $\ps$ is given by
\begin{equation}\label{eq.lpsi}
    M\ps = \frac{W^*(\f_0,\ldots,\f_{n-1},\ps)}{W^*(\f_0,\ldots,\f_{n-1})}\,,
\end{equation}
and for $n=2k$ the action of $M$ on arbitrary \emph{even} function $\f$ is given by
\begin{equation}\label{eq.lphi}
    M\f = \frac{W(\f_0,\ldots,\f_{n-1},\f)}{W(\f_0,\ldots,\f_{n-1})}\,.
\end{equation}
In each case, to functions of the different parity, $M$ is extended by linearity.
\end{theorem}
(This obviously generalizes to arbitrary non-degenerate operators. Note the emergence of inverse Wronskians for operators of odd order.)
\begin{proof} Suppose $\f_0,\ldots,\f_{n-1}$ is a basis in $\Ker M$ as stated. Then
\begin{align*}
    a_n\f_0+a_{n-1}D\f_0+\ldots + a_{1}D^{n-1}\f_0+D^n\f_0 =0\,,\\
    a_n\f_1+a_{n-1}D\f_1+\ldots + a_{1}D^{n-1}\f_1+D^n\f_1 =0\,,\\
    \ldots \\
    a_n\f_{n-1}+a_{n-1}D\f_{n-1}+\ldots + a_{1}D^{n-1}\f_{n-1}+D^n\f_{n-1} =0\,,
\end{align*}
which gives a system of linear equations for determining the coefficients $a_n,a_{n-1},\ldots,a_{1}$ with the square matrix  $\wrom(\f_0, \f_1, \ldots \f_{n-1})$  and the right-hand sides $D^n\f_0$,\ldots, $D^n\f_{n-1}$. Hence the operator $M$ is uniquely defined by  $\f_0,\ldots,\f_{n-1}$. (Explicit expressions for the coefficients can be obtained, e.g., by the `super Cramer formulas'~\cite{rabin:duke},~\cite{tv:ber}, which use the notion of cofactors introduced in the previous section.)  To deduce equalities~\eqref{eq.lpsi} and \eqref{eq.lphi}, consider the linear operators  defined by the right-hand sides of~\eqref{eq.lpsi} and \eqref{eq.lphi}. They are of  order $n$. Note  that, by the properties of super Wronskians, if a linear combination of $\f_0,\ldots,\f_{n-1}$ of the required parity (the coefficients being  treated as independent parameters some of which are odd) is substituted for $\psi$ in the right-hand side of~\eqref{eq.lpsi} or $\phi$ in in the right-hand side of~\eqref{eq.lphi}, then the right-hand sides of \eqref{eq.lpsi} and \eqref{eq.lphi} vanish. By differentiating w. r. t. the parameters, it follows  that all the functions $\f_0,\ldots,\f_{n-1}$ are in the kernels of these operators. To identify them finally with $M$, it is sufficient to see that these operators are monic. This amounts to calculating the coefficient of $D^{2k+1}\psi$ in the expansion of
\begin{equation*}
    W^*(\f_0,\ldots,\f_{2k},\ps)=
    \Ber^*\begin{pmatrix}
         \f_0 & \f_1 & \dots & \f_{2k} & \psi \\
         D\f_0 & D\f_1 & \dots & D\f_{2k} & D\psi \\
         \dots & \dots & \dots & \dots & \dots \\
        D^{2k}\f_0 & D^{2k}\f_1 & \dots & D^{2k}\f_{2k} & D^{2k}\psi \\
        D^{2k+1}\f_0 & D^{2k+1}\f_1 & \dots & D^{2k+1}\f_{2k} & D^{2k+1}\psi
       \end{pmatrix}\,,
\end{equation*}
or, similarly, of $D^{2k}\f$ in the expansion of $W(\f_0,\ldots,\f_{2k-1},\f)$. This is exactly the inverse superwronskian $W^*(\f_0,\ldots,\f_{n-1})$ or the superwronskian $W(\f_0,\ldots,\f_{n-1})$, respectively.
\end{proof}

From the properties of Berezinians, the right hand sides of ~\eqref{eq.lpsi} and \eqref{eq.lphi} are invariant under an arbitrary non-singular transformation of the basis $\f_0,\ldots,\f_{n-1}$\,, so do not depend on a choice of  basis in $\Ker M$.

\begin{example} Let $M=D+\mu$ be a first-order operator. Its solution space is one-dimensional. Suppose $\f$ is an  invertible even function such that $M\f=0$. For an  odd function $\ps$,
\begin{multline*}
    M\ps=\frac{W^*(\f,\ps)}{W^*(\f)}=
     \frac{\Ber^*\begin{pmatrix}
                    \f   &   \ps \\
                   D\f    &   D\ps\\
                  \end{pmatrix}}{\Ber^*(\f)}=
   \f\,\Ber\begin{pmatrix}
                    D\ps   & \vline &  D\f\\
                   \hline
                   \ps    & \vline & \f\\
                  \end{pmatrix}=\\
                  \f\,\frac{D\ps-D\f\cdot \f^{-1}\,\ps}{\f}=
                  \left(D-D\f\cdot \f^{-1}\right)\ps\,.
\end{multline*}
The final expression is  manifestly  linear in $\ps$ and  does not depend on the parity of $\ps$; the same formula defines the action of $M$ on functions of arbitrary parities.
\end{example}

To obtain the coefficients of an operator $$M=D^n+a_1D^{n-1}+\ldots +a_{n-1}D+a_n$$ explicitly, it is possible to use directly the Cramer rule of~\cite{rabin:duke} and \cite{tv:ber} (as noted in the proof of Theorem~\ref{thm.mfromker}) or apply the formulas for the cofactors from the previous section. This gives an alternative description of an operator in terms of its fundamental solutions.
\begin{theorem} If $n=2m+1$,
\begin{equation}\label{eq.a1}
    a_1=-\frac{1}{W(\f_0,\ldots,\f_{2m})}\Ber\begin{pmatrix}
         \f_0 & \f_1 & \dots & \f_{2m}   \\
         D\f_0 & D\f_1 & \dots & D\f_{2m}   \\
         \dots & \dots & \dots & \dots   \\
        D^{2m-1}\f_0 & D^{2m-1}\f_1 & \dots & D^{2m-1}\f_{2m}   \\
        D^{2m+1}\f_0 & D^{2m+1}\f_1 & \dots & D^{2m+1}\f_{2m}
       \end{pmatrix}\,,
\end{equation}
\begin{equation}\label{eq.a2}
    a_2=-\frac{1}{W^*(\f_0,\ldots,\f_{2m})}\Ber^*\begin{pmatrix}
         \f_0 & \f_1 & \dots & \f_{2m}   \\
         D\f_0 & D\f_1 & \dots & D\f_{2m}   \\
         \dots & \dots & \dots & \dots   \\
         D^{2m+1}\f_0 & D^{2m+1}\f_1 & \dots & D^{2m+1}\f_{2m}\\
        D^{2m}\f_0 & D^{2m}\f_1 & \dots & D^{2m}\f_{2m}   \\
       \end{pmatrix}\,,
\end{equation}
and so on,
\begin{equation}\label{eq.an}
    a_{2m+1}=-\frac{1}{W(\f_0,\ldots,\f_{2m})}\Ber\begin{pmatrix}
          D^{2m+1}\f_0 & D^{2m+1}\f_1 & \dots & D^{2m+1}\f_{2m}\\
         D\f_0 & D\f_1 & \dots & D\f_{2m}   \\
         \dots & \dots & \dots & \dots   \\
         D^{2m-1}\f_0 & D^{2m-1}\f_1 & \dots & D^{2m-1}\f_{2m}\\
        D^{2m}\f_0 & D^{2m}\f_1 & \dots & D^{2m}\f_{2m}   \\
       \end{pmatrix}\,.
\end{equation}
In words: we replace one row of the Wro\'{n}ski matrix by the row-vector $$(-D^n\f_0,\ldots,-D^n\f_{n-1})$$ and apply, alternating, $\Ber$ and $\Ber^*$.
Similarly for $n=2m$, where we have to start from $\Ber^*$ for the coefficient $a_1$.
\end{theorem}
\begin{proof} Cramer's rule in the form given in~\cite{tv:ber}.
\end{proof}

Note also the following useful formulas for differentiating superwronskians.
\begin{lemma}
\begin{equation}\label{eq.dwroneven}
    DW(\f_0,\ldots,\f_{2s})=\Ber\begin{pmatrix}
                                  \f_0 & \dots & \f_{2s} \\
                                   D\f_0 & \dots & D\f_{2s} \\
                                    \dots & \dots & \dots \\
                                    D^{2s-1}\f_0 & \dots & D^{2s-1}\f_{2s} \\
                                    D^{2s+1}\f_0 & \dots & D^{2s+1}\f_{2s}
                                \end{pmatrix}
\end{equation}
\begin{equation}\label{eq.dwronodd}
    DW(\f_0,\ldots,\f_{2s+1})=-W(\f_0,\ldots,\f_{2s+1})^2
                                \Ber^*\begin{pmatrix}
                                  \f_0 & \dots & \f_{2s+1} \\
                                   D\f_0 & \dots & D\f_{2s+1} \\
                                    \dots & \dots & \dots \\
                                    D^{2s}\f_0 & \dots & D^{2s}\f_{2s+1} \\
                                    D^{2s+2}\f_0 & \dots & D^{2s+2}\f_{2s+1}
                                \end{pmatrix}
\end{equation}
\emph{(Note `wrong parity' last rows.)} Similar formulas hold for $W^*(\f_0,\ldots,\f_r)$, with the cases of $r=2s$ and $r=2s+1$ interchanged.
\end{lemma}
\begin{proof} Use the formulas for differentiation of Berezinian (\eqref{eq.difber} and its analog for $\Ber^*$). Then apply the invariance of the Berezinian of a `wrong' matrix under elementary transformation of rows. This makes all terms in the sums zero except for one.
\end{proof}
\begin{coro} For $n=2m+1$, $a_1=DW(\f_0,\ldots,a_{2m})$ (the 1st coefficient of $M$).
\end{coro}

\section{Darboux transformations. Classification theorem}

Consider non-degenerate operators of order $m$ with the some fixed principal symbol (i.e., the top coefficient $a_0$). For simplicity let the operators be monic , i.e., $a_0=1$; the general case is similar.
\begin{de} For two monic  operators $L_0$ and $L_1$ of the same order, a \emph{Darboux transformation} $L_0\to L_1$ is given  by a monic differential operator $M$ of an arbitrary order $r$ such that it satisfies the  \emph{intertwining relation}

\begin{equation}\label{eq.intertw}
    ML_0=L_1M\,.
\end{equation}
\end{de}

\begin{remark} There is a more general notion based on   intertwining relations of the form $NL_0=L_1M$ with possibly different $N$ and $M$, see~\cite{shem:darboux2}. We do not consider it here.
\end{remark}

By definition, the order of an operator $M$ is called the  \emph{order}   of the Darboux transformation.  Note that $L_1$, if exists, is defined uniquely by $L_0$ and $M$. However, not every operator $M$ can give a Darboux transformation. As we shall see, equation \eqref{eq.intertw} is, in a sense, an overdetermined system and $M$ must satisfy compatibility conditions. It will become clear from considerations in this section, as well as from particular examples considered in Section~\ref{sec.dress}, that the problem of finding all $M$ defining    Darboux transformations $L_0\to L_1$ for a given operator $L_0$ can be viewed as a generalization of the eigenvalue problem   (the problem of finding eigenvalues and eigenvectors of a given operator).

We write $L_0\xto{M} L_1$ to denote the Darboux transformation defined by an operator $M$.   Darboux transformations can be composed and form a category: if $L_0\xto{M_{10}} L_1$ and $L_1\xto{M_{21}} L_2$, then $L_0\xto{M_{20}} L_2$ where $M_{20}:=M_{21}M_{10}$.

%$L_0\to L_1$ with  $M_{10}L_0=L_1M_{10}$ and $L_1\to L_2$ with  $M_{21}L_0=L_1M_{21}$, then we have $L_0\to L_2$ with  $M_{20}L_0=L_2M_{20}$, where $M_{20}:=M_{21}M_{10}$.

\begin{lemma}
 Every first-order Darboux transformation $L_0\to L_1$ is given by an operator $M=\Mf$, where $\f$ is an even eigenfunction of  $L_0$ with some %(even or odd)
 eigenvalue $\l$. (And conversely.)
\end{lemma}
\begin{proof}
 Let $\Mf L_0=L_1\Mf$ for some   $\Mf$. Divide $L_0$  by $\Mf$ from the right, so that $L_0=Q\Mf + f$ with some function $f$. Hence $\Mf\circ (Q\Mf + f)=L_1\Mf$. By applying both sides to $\f$, we obtain $\Mf(f\f)=0$; hence $f=\l=\const$. (Note that $\l$ may be even or odd depending on $m$.) Therefore $L_0\f=\l\f$. Conversely, if $L_0\f=\l\f$, we similarly deduce that $L_0=Q\Mf + \l$. Hence $\Mf L_0=L_1\Mf$   for $L_1:=\Mf Q + \l$\,.
\end{proof}

First-order Darboux transformations defined by  operators $M=\Mf$, where $\Mf=D-D\ln\f$,   are called \emph{elementary} Darboux transformations on the superline. (They are analogous to  transformations on the ordinary line given by $\p-\p\ln\f$.) As we  observed in the course of the proof,    all such Darboux transformations correspond to    changing order in an incomplete factorization of $L_0$,
\begin{equation*}
    L_0=Q\Mf + \l \ \to\   L_1=\Mf Q + \l\,.
\end{equation*}

\begin{remark} By requiring that the intertwining operator $M$ be monic, we exclude non-identical Darboux transformations of order zero. It is convenient to treat the corresponding notion independently. For an invertible function $g$, operators $A$ and $B$ are related by  a \emph{gauge transformation} generated by $g$ if $gB=Ag$ (as composition of operators), i.e., $B=A^{g}:=g^{-1}Ag$\,. Gauge transformations commute with Darboux transformations in the sense that if $L_0\xto{M} L_1$, then $L_0^g\xto{M^g} L_1^g$\,. Darboux transformations themselves can be interpreted as `gauge' or `similarity' transformations in the ring $\DO(1|1)[[D^{-1}]]$ of formal pseudodifferential operators with the restriction that the image of $L_0$ is a differential operator.
\end{remark}

\begin{theorem}\label{main}
 Every Darboux transformation $L_0\xto{M} L_r$ of order $r$ is the composition of $r$   elementary first-order transformations:
 \begin{equation}\label{eq.compos}
    L_0\xto{M_{\f_1}} L_1\xto{M_{\f_2}} L_2\xto{M_{\f_3}} \ldots \xto{M_{\f_r}} L_r\,.
 \end{equation}
\end{theorem}
\begin{proof}
 Suppose $ML_0=L_rM$ for an operator $M$ of order $r$. It can be decomposed into first-order factors, $M=M_1\cdot\ldots\cdot M_r$, where $M_i=M_{\f_i}$ for some functions $\f_i$. Note that this does not suffice \emph{per se}, because we would also need to find the intermediate operators $L_1, \ldots, L_{r-1}$ related by the corresponding  Darboux transformations. We proceed by induction. Suppose $r>1$. From the intertwining relation, $L_0(\Ker M)\subset \Ker M$. Recall that $\dim \Ker M=s|s \ \text{or} \ s+1|s$, if $r=2s \ \text{or} \  r=2s+1$. Take an invertible eigenfunction $\f$ of $L_0$ in $\Ker M$. Then $M=M'\Mf$, where $M'$ is a non-degenerate operator of order $r-1$, and $\Mf$ defines a Darboux transformation $L_0\to L_1$. We shall show that $M'$ defines a Darboux transformation $L_1\to L_r$. Indeed, from $M'\Mf L_0=L_rM'\Mf$ and  $\Mf L_0=L_1\Mf$, we obtain $M'L_1\Mf=L_rM'\Mf$. This implies $M'L_1=L_rM'$, as $\Mf$ is a non-zero-divisor. Thus $L_0\xto{M}L_r$ is factorized into $L_0\xto{\Mf}L_1\xto{M'}L_r$ where $\ord M'=r-1$, and this completes the inductive step.
\end{proof}

\begin{remark} An analog of Theorem~\ref{main} holds true  for arbitrary  operators on the ordinary   line, where elementary transformations   are specified by operators of the form $\partial-\partial\ln\varphi$. A proof can be found in~\cite{adler-marikhin-shabat:2001}. Before that, it took some effort to establish this result for the particular case of  Sturm--Liouville operators~\cite{veselov-shabat:dress1993}, \cite{bagrov-samsonov:factorization1995}, \cite{samsonov:factorization1999}. Note that in the fundamental monograph~\cite{matveev-salle:book1991}, Darboux transformations are  introduced by definition as  iterations of elementary transformations.
\end{remark}

\begin{comment}
For 2D operators,   the more general intertwining relation $NL_0=L_1M$ has to be used. The  theory in 2D is richer because of different types of Darboux transformations (such as Laplace transformations besides Wronskian transformations).
Factorization of Darboux transformations for 2D Schr{\"o}dinger operator conjectured by Darboux was established in~\cite{shemy:complete,shemy:fact};  new  invertible transformations were found in~\cite{shemy:invert}. We hope to study   elsewhere the  attractive possibilities that may open in the super version.
\end{comment}

It is possible to give a closed form for the composition of elementary Darboux transformations in Theorem~\ref{main}, as follows. As we know (Theorem~\ref{thm.mfromker}), the $r$th order  operator $M$ specifying a Darboux transformation $L_0\to L_r$ can be reconstructed from its kernel. To this end, we need to find a basis in $\Ker M$. For convenience, let us re-write the factorization given by Theorem~\ref{main} as
\begin{equation*}
    M=M_{\phi_{r-1}}M_{\phi_{r-2}}\ldots M_{\phi_0}
\end{equation*}
(we have changed the notation), where all functions $\phi_i$ are even. We shall construct  a basis $\f_0,\f_1,\ldots,\f_{r-1}$ in $\Ker M$, where $\widetilde{\f_i}=i\mod 2$, as follows\footnote{We use different lettershapes $\phi$ and $\f$ of the same Greek letter phi. (A delight for Greek-lovers.)}. We
 already have $\phi_0\in\Ker M$; we set $\f_0:=\phi_0$. Set $\f_1$ to be an arbitrary element in the preimage of $\phi_1$ under the operator $M_{\phi_0}$; it is defined up to a multiple of $\phi_0$. Since $\phi_1$ is even, the function $\f_1$ has to be odd. By construction, the function $\f_1$ is annihilated by the operator $M_{\phi_1}M_{\phi_0}$ and hence by $M$. On the other hand, $\phi_1$ is an eigenfunction of the operator $L_1$ with some eigenvalue $\l_1$, so we have  $L_1\phi_1=\l_1\phi_1$; and  from the intertwining relation $M_{\phi_0}L_0=L_1M_{\phi_0}$ we have
 \begin{multline*}
    M_{\phi_0}\bigl(L_0\f_1-\l_1\f_1\bigr)=M_{\phi_0}L_0\f_1-\l_1M_{\phi_0}\f_1= L_1M_{\phi_0}\f_1-\l_1M_{\phi_0}\f_1=L_1\phi_1-\l_1\phi_1=0\,.
 \end{multline*}
 (An extra sign can emerge for an odd $L_0$.)
 That means that $L_0\f_1=\l_1\f_1 (\mod \f_0)$\,. Similarly we can introduce an even function $\f_2$ as an element in the preimage of $\phi_2$ under the operator $M_{\phi_1}M_{\phi_0}$, which is defined up to a linear combination of $\f_0$ and $\f_1$. It will be annihilated by $M_{\phi_2}M_{\phi_1}M_{\phi_0}$ and hence by $M$. Also, $\f_2$ will be an `eigenfunction for $L_0$  up to'  the linear span  of $\f_0$ and $\f_1$, with the eigenvalue $\l_2$ (the eigenvalue of $L_2$ corresponding to the eigenfunction $\phi_2$). And so on. In this way we construct functions $\f_0,\f_1,\ldots,\f_{r-1}$ of alternating parities which make a basis  of $\Ker M$ and such that the restriction of $L_0$ on its invariant subspace $\Ker M$ is triangular in this basis.  Note that  $\f_0,\f_1,\ldots,\f_{r-1}$ are not necessarily     eigenfunctions of $L_0$.

 We have established the following theorem.

 \begin{theorem} Every Darboux transformation of order $r$ between non-degenerate differential operators on the superline $L_0\to L_r$  is specified by an invariant subspace of $L_0$ of dimension $s+1|s$ if $r=2s+1$ or $s|s$ if $r=2s$, to which uniquely corresponds
 a monic operator $M$ so that the intertwining relation $ML_0=L_rM$ holds; if $\f_0,\ldots,\f_{r-1}$ is a basis of this subspace consisting of functions of parities $\widetilde{\f_i}=i\mod 2$, then $M$ is given by the super Wronskian  formula
 \begin{equation}\label{eq.lpsi2}
    M\ps = \frac{W^*(\f_0,\ldots,\f_{r-1},\ps)}{W^*(\f_0,\ldots,\f_{r-1})}\,,
\end{equation}
for odd functions $\ps$, if $r=2s+1$, or
\begin{equation}\label{eq.lphi2}
    M\f = \frac{W(\f_0,\ldots,\f_{r-1},\f)}{W(\f_0,\ldots,\f_{r-1})}\,,
\end{equation}
for even functions $\f$, if $r=2s$, and on functions of the other parities, the operator $M$ is extended by linearity.
 \end{theorem}

 Note that we have a one-to-one correspondence between Darboux transformations and invariant subspaces. A one-dimensional invariant subspace is an eigenspace and corresponds to an elementary Darboux transformation. Theorem~\ref{main} guarantees that every Darboux transformation factorizes into elementary. Suppose, however, we did not know that; then   looking for Darboux  transformations directly, by solving the intertwining relation~\eqref{eq.intertw} with respect to $L_1$, would lead us to a generalization of the eigenvalue problem. This is seen in the examples in the next section.

 \section{Examples of dressing transformations}\label{sec.dress}

Let $L_0\to L_1$ be a Darboux transformation given by an intertwining operator $M$. The expression of (the coefficients of) the operator $L_1$ in terms of given $L_0$ and $M$ is often called the \emph{dressing} (sometimes `undressing') \emph{transformation} for $L_0$. We shall calculate dressing transformations explicitly for arbitrary (monic) operators on the superline of orders $\leq 4$. To work them out in these concrete examples, nothing is required in principle but patience. But it is elucidating to see   all the theoretical considerations above become in calculations completely manifest. Indeed, suppose $L_0$ is of order $n$,
\begin{equation}\label{eq.lgen}
    L_0=D^n+a_1D^{n-1}+a_2D^{n-2}+\ldots+a_{n-1}D +a_n\,,
\end{equation}
and we write $L_1$ as
\begin{equation}\label{eq.lgendress}
    L_1=D^n+b_1D^{n-1}+b_2D^{n-2}+\ldots+b_{n-1}D +b_n\,.
\end{equation}
Suppose the intertwining operator is of order $r$\,:
\begin{equation}\label{eq.mgen}
    M=D^r+c_1D^{r-1}+c_2D^{r-2}+\ldots+c_{r-1}D +c_r\,.
\end{equation}
(In the chosen notation, the parity of each coefficient such as $a_k$ is   $k  (\mod  2)$, which is convenient for calculations.) Practically, one needs to find the coefficients $b_k$ in terms of $a_i$ and $c_j$. These are the desired dressing formulas. In principle, all that one has to do to find the dressing formulas, is to write down the relation~\eqref{eq.intertw}, substitute into it the expansions~\eqref{eq.lgen}, \eqref{eq.mgen} and \eqref{eq.lgendress}, calculate the compositions and compare the coefficients. This would give the expression for $b_k$. It is worth noting, however, that there are many more equations than we need (the system is overdetermined). As we shall see, the first $n$ equations (linear in $b_k$) allow to determine the coefficients $b_k$, where $k=1,2, \ldots, n$, recursively, where the remaining equations give non-linear compatibility conditions for the coefficients $c_i$. (For $r=1$, one obtains a Riccati-type equation, which  by a logarithmic derivative substitution becomes an equation for eigenfunctions.) Since the coefficients of $M$ must satisfy compatibility conditions, they are not independent. Therefore dressing formulas are non-unique. However, the described procedure gives them probably in the  simplest form.

After expansion, the intertwining relation $ML_0=L_1M$ takes the form
\begin{multline}\label{eq.intertwexpand}
    (D^r+c_1D^{r-1}+ \ldots+c_{r-1}D +c_r)\circ (D^n+a_1D^{n-1}+ \ldots+a_{n-1}D +a_n)= \\
    (D^n+b_1D^{n-1}+ \ldots+b_{n-1}D +b_n)\circ (D^r+c_1D^{r-1}+ \ldots+c_{r-1}D +c_r)\,.
\end{multline}
To open the brackets in~\eqref{eq.intertwexpand}, one needs to be able to calculate the compositions of the form $D^k\circ f$, where $f$ is a function. One should use the formula (the Leibniz formula for the superline):
\begin{equation}\label{eq.leibn}
    D^k(fg)=\sum_{p+q=k}\sbinom{k}{p}(-1)^{q\ft}D^pf\cdot D^qg\,.
\end{equation}
Here $\sbinom{k}{p}$ is the \emph{superbinomial coefficient} introduced by Manin and Radul~\cite{manin:radul}. Formula~\eqref{eq.leibn} defines a diagonal (coproduct) on operators with constant coefficients on the superline, which form a Milnor--Hopf algebra with respect to the ordinary composition and this coproduct. It is easy to find that
\begin{equation}\label{eq.sbinom}
\sbinom{2s+1}{2t+1}=\sbinom{2s+1}{2t}=\binom{s}{t}\,, \qquad \sbinom{2s}{2t+1}=0\,,  \qquad \sbinom{2s}{2t}=\binom{s}{t}\,.
\end{equation}
In particular, $\sbinom{2s+1}{1}=1$ and $\sbinom{2s}{1}=0$\,.
Also,
\begin{equation*}
    \sbinom{k}{p}=\sbinom{k}{k-p}\,,
\end{equation*}
as for ordinary binomial coefficients. To find the dressing transformations from~\eqref{eq.intertwexpand}, one needs to calculate the products in the both sides  modulo operators of order $\leq r-1$. This would give exactly $n$ equations   defining the $n$ coefficients $b_1$, \dots, $b_n$. The lower order terms give  compatibility conditions for $c_i$.

\begin{remark} For operators on the ordinary line, $L=\p^n+a_1\p^{n-1}+\ldots$, it is well known that the first coefficient $a_1$ is a `pure gauge', i.e., by a gauge transformation can be made zero. This is not the case for operators on the superline. If $n=2m$, the coefficient $a_1$ in $L=D^{2m}+a_1D^{2m-1}+\ldots $ is gauge-invariant. (The explanation is in the equality   $\sbinom{2s}{1}=0$.)
\end{remark}

\begin{example} Consider the operator
$$L_0=D+a_1\,.$$
We can still learn something from this seemingly trivial example. By formally calculating the dressing transformation, we obtain
\begin{equation*}
    b_1 =(-1)^r a_1+2c_1\,.
\end{equation*}
To understand it better, consider  an elementary transformation with $M=D+\mu$. By fully expanding $(D+\mu)(D+a_1)=(D+b_1)(D+\mu)$, we obtain $b_1=-a_1+2\mu$, which is a particular case of the above, and also $D\mu=Da_1$, which is the compatibility condition.  It gives $\mu=a_1+\la$, where $\la$ is an odd constant. That means that $M=D+a_1+\la$ indeed corresponds to an eigenfunction of $L_0$ (with an odd eigenvalue), as it should be according to general theory. As for the dressing transformation, we have $b_1 =-a_1+2(a_1+\la)=a_1+2\la$, i.e., it is simply a constant shift.
\end{example}

\begin{example} Consider  the operator
$$L_0=D^2+a_1D+a_2=\p +a_1D +a_2\,.$$
The dressing transformation is
\begin{equation*}
    \left\{
    \begin{aligned}
    b_1&=(-1)^ra_1\,,\\
    b_2&=a_2-(-1)^r\sbinom{r}{1}Da_1-2(-1)^rc_1a_1\,.
       \end{aligned}
\right.
\end{equation*}
Hence for odd $r=2s+1$,
\begin{equation*}
    \left\{
    \begin{aligned}
    b_1&=-a_1\,,\\
    b_2&=a_2+Da_1+2c_1a_1\,,
       \end{aligned}
\right.
\end{equation*}
and for even $r=2s$,
\begin{equation*}
    \left\{
    \begin{aligned}
    b_1&= a_1\,,\\
    b_2&=a_2-2c_1a_1\,.
       \end{aligned}
\right.
\end{equation*}
If $a_1=0$, $L_0=\p+a_2$ is an `ordinary' operator (of first order in the usual sense). We observe that the condition $a_1=0$ is preserved by Darboux transformations; moreover, in this case, $b_2=a_2$, i.e., the operator $L_0=\p+a_2$ does not have non-trivial dressing transformations. (All Darboux transformations  are ``autotransformations'' $L_0\to L_0$.)
\end{example}

\begin{example} Consider  the operator
$$L_0=D^3+a_1D^2+a_2D+a_3=D\p+a_1\p +a_2D +a_3\,.$$
The dressing transformation is
\begin{equation*}
    \left\{
    \begin{aligned}
    b_1&=(-1)^ra_1+2c_1\,,\\
    b_2&=a_2-(-1)^r\sbinom{r}{1}Da_1-Dc_1\,, \\
    b_3&=
    (-1)^ra_3+\sbinom{r}{1}Da_2+(-1)^r\sbinom{r}{2}\p a_1+
    2c_1a_2+\\
    & \hphantom{MMMMMMMMMM} Da_1c_1+(-1)^rc_2a_1+
    \p c_1 -Dc_2 +2c_3-Dc_1 c_1
  \,.
       \end{aligned}
\right.
\end{equation*}
\end{example}

\begin{example} Consider the operator
$$L_0=D^4+a_1D^3+a_2D^2+a_3D+a_4= \p^2+a_1D\p +a_2\p+a_3D+a_4\,.$$
When $a_1=0=a_2$, this is the super Sturm--Liouville operator. Let us first analyze the general case. For simplification, consider an elementary Darboux transformation given by $M=D+\mu$. The corresponding dressing transformation is
\begin{equation*}
    \left\{
    \begin{aligned}
    b_1&=-a_1\,,\\
    b_2&=a_2 + Da_1+2\mu a_1\,, \\
    b_3&=-a_3+ Da_2+D(\mu a_1)-2\p\mu\,,\\
    b_4&=a_4+ Da_3+ 2\mu a_3-\mu Da_2+\p\mu a_1-(D\mu \mu)a_1-2\p\mu \mu\,.\\
       \end{aligned}
\right.
\end{equation*}
Notice the following. If $a_1=0$, this condition is preserved by Darboux transformations. If $a_1=0$ holds, then $a_2$ is preserved by Darboux transformations. The term $a_2\p$ is a `magnetic' term, and
$$L_0= \p^2 +a_2\p+a_3D+a_4$$
is the `super Sturm--Liouville operator in magnetic field'. The magnetic term is however a pure gauge: it can be made zero by a gauge transformation $L\to g^{-1}Lg$ with a suitable $g$. (Recall that gauge transformations commute with Darboux transformations.)
\end{example}

\begin{example} The previous analysis justifies consideration of the `pure' super Sturm--Liouville operator (without magnetic term)
$$L_0= \p^2+\al D+u$$
and explains why its form is invariant under Darboux transformations. Let
$$L_1=\p^2+\be D +w\,.$$
The   dressing transformation (for a general intertwining operator $M$ of order $r$) is given by
\begin{equation*}
    \left\{
    \begin{aligned}
    \be&=(-1)^r\al+2\p c_1\,,\\
    w&=u -(-1)^r\sbinom{r}{1}D\al -2(-1)^rc_1\al -2\p c_1 c_1 -2\p c_2\,.\\
       \end{aligned}
\right.
\end{equation*}
Darboux transformations for the super Sturm--Liouville operator were studied in~\cite{liu-qp-darboux-for-skdv-lmp-1995}, \cite{liu-manas-darboux-for-manin-radul-1997,liu-manas-crum-skdv-1997}, and the most complete formulas were obtained in \cite{li-nimmo:darboux-twist2010}; by comparing notations, one may conclude that the formulas above are the same as those in~\cite{li-nimmo:darboux-twist2010}.
\end{example}

The coefficients $c_k$ in the dressing transformations can be expressed in terms of super Wronskians as in Section~\ref{sec.swron}.

Examples above show some patterns in dressing transformations. They can be summarized in the following simple theorem.

\begin{theorem} Consider a general monic operator $L_0$ of order $n$ given by equation~\eqref{eq.lgen}. The following holds for its dressing transformations \emph{(we show formulas for elementary transformations)}:
\begin{itemize}
  \item If $n=2m+1$, there are simple dressing formulas for the first two coefficients:
   \begin{equation*}
    \left\{
    \begin{aligned}
    b_1&=-a_1+2\mu\,,\\
    b_2&=a_2+Da_1-D\mu\,;\\
       \end{aligned}
\right.
\end{equation*}
  \item If $n=2m$, then under dressing transformations
  \begin{equation*}
    \left\{
    \begin{aligned}
    b_1&=-a_1\,,\\
    b_2&= a_2+Da_1-2\mu a_1\,,\\
       \end{aligned}
\right.
\end{equation*}
hence the condition $a_1=0$ is invariant and, if it holds, the coefficient $a_2$ is preserved.
\end{itemize}
\end{theorem}
We skip the proof. The dressing formulas are deduced exactly in the same way as in the examples above.

%\bibliographystyle{plain}
%%\bibliography{/home/ted_voronov/Desktop/Work/Local_TeX_Files/bibtex/bib/misc/geometry} %% Office UNIX
%\bibliography{geometry,general} %% Laptop Windows
%\end{document}

\def\cprime{$'$}

\end{document}